\documentclass[11pt,a4]{amsart}

\usepackage{amsmath}
\usepackage{cancel}
\usepackage{graphicx}
\usepackage{url,hyperref}
\usepackage[style=american]{csquotes}
\usepackage{enumerate}
\usepackage[width=\textwidth]{caption}
\usepackage{subcaption}
\DeclareCaptionSubType*[arabic]{figure}
\usepackage[normalem]{ulem}

\usepackage{own-macros}

\begin{document}

\title[Minkowski concentricity and complete simplices]{Minkowski concentricity and complete simplices}

\author{Ren\'e Brandenberg and Bernardo Gonz\'alez Merino}
\address{Zentrum Mathematik, Technische Universit\"at M\"unchen, Boltzmannstr. 3, 85747 Garching bei M\"unchen, Germany} \email{brandenb@ma.tum.de} \email{bg.merino@tum.de}

\thanks{The second author is partially supported by Consejer\'ia de Industria, Turismo, Empresa e Innovaci\'on de la CARM through Fundaci\'on S\'eneca, Agencia de Ciencia y Tecnolog\'ia de la Regi\'on de Murcia, Programa de Formaci\'on Postdoctoral de Personal Investigador 19769/PD/15 and project 19901/GERM/15, Programme in Support of Excellence Groups of the Regi\'{o}n de Murcia, and by MINECO project reference MTM2015-63699-P, Spain.}



\begin{abstract}
This paper considers the radii functionals (circumradius, inradius, and diameter) as well as the Minkowski asymmetry for
general (possibly non-symmetric) gauge bodies. A generalization of the concentricity inequality (which states that the sum of the inradius and circumradius is not greater than the diameter in general Minkowski spaces) for non-symmetric gauge bodies is derived and a strong connection
between this new inequality, extremal sets of the generalized Bohnenblust inequality, and completeness of simplices is revealed.
\end{abstract}

\maketitle

\section{Introduction} \label{s:intro}

We call any convex, compact set a \cemph{dred}{(convex) body} and denote by $\CK^n$ the family of convex bodies in $\R^n$ (excluding singletons).
For any $K,C \in \CK^n$ the \cemph{dred}{inradius} $r(K,C)$ of $K$ \wrt~$C$ is
the largest $\lambda \ge 0$ such that $\lambda C$ is contained in a translation of $K$, while
the \cemph{dred}{circumradius} $R(K,C)$ of $K$ \wrt~$C$ is the smallest $\lambda > 0$ such that $\lambda C$
contains a translate of $K$. Defining the \cemph{dred}{length} of a segment $[x,y]$ \wrt~$C$ with endpoints $x,y \in \R^n$ by $2R([x,y],C)$, the \cemph{dred}{diameter} $D(K,C)$ of $K$ \wrt~$C$ is the maximal length of a segment with both endpoints in $K$.
For arbitrary $C \in \CK^n$ the \cemph{dred}{Minkowski asymmetry} $s(C)$ of $C$ is the smallest
$\lambda>0$ such that $-C\subset c+\lambda C$ for some $c\in\R^n$, \ie~$s(C)=R(-C,C)$.
Asymmetry functions are often used to extend and unify results, which have natural solutions when restricted to
  centrally symmetric bodies as well as for the general case (see, \eg \cite{BeLa}, \cite{BrKo2}, \cite{GuKa}, \cite{toth2015Measures},  for such results using the Minkowski asymmetry and \cite{Chak}
  using other symmetry measures).

The \cemph{dred}{Jung ratio} of $K$ \wrt~$C$ is the quotient $j(K,C) := R(K,C)/D(K,C)$ between the
circumradius and the diameter, while the \cemph{dred}{Jung constant} $j_C$ of a body $C$ is defined as
$j_C:=\max\set{j(K,C):\,K\in\CK^n}$. Both namings honour Jung, who proved that $j_{\B_2^n}=\sqrt{\frac{n}{2(n+1)}}$
  for the Euclidean ball $\B_2^n$ \cite{Ju}.

A general upper bound for the Jung constant of a symmetric body $C$ is given by the inequality of Bohnenblust \cite{Bo}:
\begin{equation} \label{eq:Bohnenblust}
  j_C \le \frac{n}{n+1}.
\end{equation}

In \cite[Theorem 4.1]{BrKo2} a generalization of Bohnenblust's inequality \eqref{eq:Bohnenblust} for arbitrary $C$ is given, involving the Minkowski asymmetry of $K$ and $C$:

For any $K,C\in\CK^n$ it holds
\begin{equation} \label{eq:ExtBohnenblust}
 j(K,C) \le \frac{s(K)(s(C)+1)}{2(s(K)+1)}
\end{equation}
which results into
\begin{equation} \label{eq:BrKo1}
  j_C \le \frac{n(s(C)+1)}{2(n+1)}, %\le \frac n 2.
\end{equation}
for the Jung constants (which has already been derived in \cite{BrRo}).

Denoting the $n$-dimensional Euclidean unit ball by $\B_2^n$, Santal\'o (in case of $n=2$) \cite{Sa} and later Vr\'ecica \cite[Corollary 1]{Vr} showed that $r(K,\B_2^n)+R(K,\B_2^n) \le D(K,\B_2^n)$.

The above inequality
appears frequently in the literature (see \eg~\cite{BrDaGrLa} and \cite{HeCiPaSaSe}) and has been generalized to arbitrary Minkwoski spaces \cite{CasPap} (and even to general Banach spaces \cite{CasPap,MoPaPh2}), resulting in the
\cemph{dred}{concentricity inequality}
(the name pointing out that for the equality case it is necessary that every incenter is also a circumcenter --
see \cite[Lemma 2.4]{BrGo}):
For all $K,C \in \CK^n$ with $C$ (centrally) symmetric it holds
\begin{equation} \label{eq:r+R<D}
  r(K,C)+R(K,C) \le D(K,C).
\end{equation}
Moreover, in \cite[Remark 6.3]{BrKo2} the concentricity inequality has been embedded into a
chain of inequalities for symmetric $C$, involving (besides others)
the Minkowski asymmetry of $K$ and the extended Bohnenblust inequality
\eqref{eq:ExtBohnenblust} with $s(C)=1$:
% \BSays{An extended formulation of the concentricity inequality \eqref{eq:r+R<D} was shown in \cite[Remark 6.3]{BrKo2} whenever $C=-C$:}
\begin{equation}\label{eq:chainMinkowski}
  (1+s(K))r(K,C)\leq r(K,C)+R(K,C)\leq\frac{1+s(K)}{s(K)}R(K,C)\le D(K,C).
\end{equation}

In this paper we give two different generalizations of the inequality chain in \eqref{eq:chainMinkowski}
  for non-symmetric $C$:
\begin{thm} \label{th:sr+R<=s+12D}
  Let $K,C\in\CK^n$. Then the two following inequality chains hold true:
  \begin{equation}\label{eq:sr+R<=s+12D}
    \begin{split}
      (1+s(K))r(K,-C) &\le r(K,-C) + R(K,C) \\
      & \le s(C)r(K,C) + R(K,C) \le \frac12(1+s(C))D(K,C)
    \end{split}
  \end{equation}
  and
  \begin{equation}\label{eq:(1+s)/sR<=s+12D}
    \begin{split}
      (1+s(K))r(K,-C) &\le r(K,-C) + R(K,C) \\
      &\le \frac{1+s(K)}{s(K)} R(K,C) \le \frac12(1+s(C))D(K,C).
    \end{split}
  \end{equation}
\end{thm}

Let us remark that for general $K$ the two chains do not fit together.

One should also recognize that the chains include two possible generalizations of the concentricity inequality
to the non-symmetric case, the \cemph{dred}{mirrored concentricity inequality}
\begin{equation} \label{eq:pm-concentricity}
  r(K,-C) + R(K,C) \le \frac12(1+s(C))D(K,C)
\end{equation}
and the stronger \cemph{dred}{generalized concentricity inequality}
\begin{equation} \label{eq:pp-concentricity}
  s(C)r(K,C) + R(K,C) \le \frac12(1+s(C))D(K,C).
\end{equation}

We will see that proving the generalized concentricity inequality is the main ingredient of
proving Theorem \ref{th:sr+R<=s+12D} and that behind the extremality of this inequality lies a
generalized notion of concentricity for two possibly non-symmetric sets.
%\RSays{However, even though proving all inequalities within the chain \eqref{eq:sr+R<=s+12D} suffices to prove the
%mirrored concentricity inequality, we will also present a direct prove of it to point out its more simple geometric interpretation -- DO %WE?.}\BSays{I did not find it so far :D} \RSays{We should keep trying.}

Given $K,C\in\CK^n$, we say that $K$ is \cemph{dred}{(diametrically) complete} \wrt~$C$ if
any strict superset of $K$ has strictly bigger diameter \wrt~$C$ than $K$. Moreover, we say that $K^*$ is a \cemph{dred}{completion} of $K$ if
$K^*$ is complete, a superset of $K$, and has the same diameter \wrt~$C$ than $K$.
The history of the concentricity inequality is closely related to efforts in understanding complete sets and completions.
Indeed it is known that completions always exist and that complete sets fulfill
equality in the concentricity inequality (see  \cite{Egg} for the Euclidean case, \cite{Sallee}
for abitrary Minkowski spaces,
as well as \cite{BaPa} and \cite[Proposition 6]{CasPap} for a generalization to Banach spaces).
Actually it is shown in \cite{BrGo2} that equality holds for the full chain in \ref{eq:chainMinkowski} if
$K$ is complete \wrt~$C$.
Thus for a general $K$ and a completion $K^*$ of $K$ it holds
$r(K,C)+R(K,C) \le r(K^*,C)+R(K^*,C) = D(K^*,C) = D(K,C)$
(\cf~\cite{Vr}) proving this way the validity of the concentricity inequality (\cf~also \cite{MoSc}).

It is easy to see that for the two inequality chains in Theorem \ref{th:sr+R<=s+12D}
equality between any part and the rightmost for some $K$ and $C$ can also be kept when completing $K$.
Moreover, for complete sets the two inequality chains in  Theorem \ref{th:sr+R<=s+12D} can be joined into
one chain:
\begin{thm} \label{thm:complete-chain}
  Let $K,C \in \CK^n$ such that $K$ is complete \wrt~$C$. Then
  \begin{equation}\label{eq:complete chain}
    \begin{split}
        (1+s(K))r(K,-C) &\le r(K,-C) + R(K,C) \le \frac{1+s(K)}{s(K)} R(K,C) \\
      & \le s(C)r(K,C) + R(K,C) \le \frac12(1+s(C))D(K,C)
    \end{split}
  \end{equation}
  holds true.
\end{thm}

The equality case of \eqref{eq:Bohnenblust} has been first studied in \cite{Le}  and extended in \cite[Corollary 2.10]{BrGo}:

\begin{prop}\label{prop:ScompleteC-C}
Let $S,C\in\CK^n$ be such that  $S$ is an $n$-simplex and $C$ symmetric. Then the following are equivalent:
\begin{enumerate}[(i)]
\item $S-S\subset D(S,C)C\subset(n+1)((S-c)\cap(-S+c))$, for some $c\in\R^n$,
\item $S$ fulfills the full chain \eqref{eq:chainMinkowski} with equality,
\item $S$ fulfills \eqref{eq:r+R<D} with equality,
\item $S$ fulfills \eqref{eq:Bohnenblust} with equality, and
\item $S$ is complete \wrt~$C$.
\end{enumerate}
\end{prop}

While for symmetric $C$ the full chain \eqref{eq:chainMinkowski} is fulfilled with equality whenever $K$ is
  complete \wrt~$C$, in the general case there exist examples of $K$ complete \wrt~$C$,
  showing that none of the inequalities in \eqref{eq:complete chain} must hold with equality. On the one hand for example,
  if we have $K=C-C$, then we may easily calculate that
   \begin{equation*}
    \begin{split}
        (1+s(K))r(K,-C) &< r(K,-C) + R(K,C) < \frac{1+s(K)}{s(K)} R(K,C) \\
      & = s(C)r(K,C) + R(K,C) = \frac12(1+s(C))D(K,C),
    \end{split}
  \end{equation*}
  whereas on the other hand Example \ref{example:ScompleteBUTnotpseudocomplete} will give a pair $K$ complete \wrt~$C$
  such that
  \begin{equation*}
    \begin{split}
      (1+s(K))r(K,-C) &= r(K,-C) + R(K,C) = \frac{1+s(K)}{s(K)} R(K,C) \\
      & < s(C)r(K,C) + R(K,C) < \frac12(1+s(C))D(K,C).
    \end{split}
  \end{equation*}

In the following we write $K \subset_t C$ (resp.~$K =_t C$) for any two convex sets $K,C$ to stress that there exists a
translation vector $c$ such that $K \subset c+C$ (resp. $K=c+C$) and abbreviate by
$K \subset^{\opt} C$ that $K \subset C$, but $K \not \subset_t \rho C$ for any $\rho<1$.

Studying the equality case of \eqref{eq:BrKo1} results in the following extension of Proposition \ref{prop:ScompleteC-C}
to arbitrary $C$:
\begin{thm}\label{th:Schain}
  Let $S,C\in\CK^n$ be such that $S$ is an $n$-simplex. Then the following are equivalent:
  \begin{enumerate}[(i)]
  \item \label{(iii)schain} $(n+1)/n \ S \subset_t S-S \subset \frac{D(S,C)}{2}(C-C) \subset_t (s(C)+1)\frac{D(S,C)}{2}C\subset_t(n+1)(-S)$, and
  \item \label{(iv)schain} $S$ fulfills the full chain \eqref{eq:sr+R<=s+12D} as well as the
    full chain \eqref{eq:(1+s)/sR<=s+12D} (and therefore the joint chain \eqref{eq:complete chain})
    with equality,
  \item \label{(i)schain} $S$ fulfills \eqref{eq:pp-concentricity} with equality,
  \item \label{(ii)schain} $S$ fulfills \eqref{eq:BrKo1} with equality,
  \item \label{(v)schain} $S$ is complete \wrt~$C$ and $R(S,C)=ns(C)r(S,C)$.
  \end{enumerate}
\end{thm}

Observe that the leftmost inclusion in Part \eqref{(iii)schain} above is always true (and thus could have been added to the chain of inclusion in Part (i) of Proposition \ref{prop:ScompleteC-C}, too). However,
we only mention it here since it holds $(n+1)/n \ S \subset_t^{\opt} (n+1)(-S)$.

One should also observe that the chain of inclusion under translations in Part \eqref{(iii)schain} above becomes a direct chain of inclusions if
and only if $S$ and $C$ have the origin as a common Minkowski center.

\section{Preliminaries}\label{s:notation}

For any $A\subset \R^n$ we denote by $\bd(A)$ the \cemph{dred}{boundary} of $A$,
%We write %$\lin(A)$,$\aff(A)$, , and $\pos(A)$
by $\conv(A)$ the \cemph{dred}{convex hull} of $A$
%and \cemph{dred}{positive hull} of $A$, respectively,
and abbreviate by $[x,y] := \conv(\set{x,y})$ the line segment with endpoints $x,y\in\R^n$.

For any $\rho>0$ and $A, B\subset\R^n$ let $A+B :=\set{a+b:a\in A,\,b\in B}$ denote the
\cemph{dred}{Minkowski sum} of $A$ and $B$ and $\rho A := \set{\rho a : a\in A}$ the \cemph{dred}{$\rho$-dilatation} of $A$,
abbreviating $-A:=(-1)A$.

%A Minkowski space $\M^n = (\R^n,\norm)$ is an $n$-dimensional real space endowed with a norm $\norm$ and
%We use $\B$ to denote the unit ball of a Minkowski space $\M^n$ and in case the

%We often abbreviate $[n] := \set{1,2,\dots,n}$.% and use $u^i \in \R^n$, $i \in[n]$, to denote the $i$-th unit vector \BSays{((We do never use unit vectors))}.

The \cemph{dred}{support function} of a convex body $K\in\CK^n$ is the function
$h(K,\,\cdot\,):\R^n \rightarrow \R$, $h(K,a)=\sup_{x \in K} a^Tx$ and
the \cemph{dred}{normal cone} of $C\subset\R^n$ in a point $p\in C$ is the set
$N(C,p)=\set{a\in\R^n:a^Tp=h(C,a)}$.
We write $H_{a,b}:=\set{x\in\R^n:a^Tx=b}$ for the
\cemph{dred}{hyperplane} orthogonal to $a\in\R^n$ with value $b\in\R$, respectively.
%We say that $p\in\bd(K)$ is \cemph{dred}{regular} (or \cemph{dred}{smooth}) if $N(K,p)=\pos(\set{a})$, for some $a\in\R^n$.
\medskip

The circumradius of $K \in \CK^n$ with respect to $C\in\CK^n$ can be formalized as
$R(K,C)=\inf\set{\rho > 0 : K \subset_t \rho C}$ (with $R(K,C)= \infty$ if and only
  if the affine hull of $K$ is not a subset of the affine hull off $C$).
Analogously, the inradius is $r(K,C) = \sup\set{\rho \ge 0 : \rho C \subset_t K}$ and we have that
$r(K,C) = R(C,K)^{-1}$.

The following Proposition is taken from \cite[Theorem 2.3]{BrKo} and gives a
characterization for $K\subset^{\opt}C$. Here and below we use $[n]$ to abbreviate $\set{1,\dots,n}$ for some $n \in \N$.

\begin{prop}[Optimal Containment Condition]\label{prop:occ}
Let $K,C\in\CK^n$ and $K\subset C$. Then the following are equivalent:
\begin{enumerate}[(i)]
\item $R(K,C)=1$.
\item there exist $i\in\set{2,\dots,n+1}$, $p^j\in K\cap \bd(C)$,
and $a^j\in N(C,p^j)$, $j\in[i]$ such that $0\in\conv(\set{a^j:j\in[i]})$.
\end{enumerate}
Moreover, if $C=\B_2^n$ then $(i)$ and $(ii)$ are also equivalent to
\begin{enumerate}[(i)]
\item[(ii')] there exist $i\in\set{2,\dots,n+1}$ and $p^j\in K\cap \bd(C)$, $j\in[i]$, such that $0\in\conv(\set{p^j:j\in[i]})$.
\end{enumerate}
\end{prop}

The \cemph{dred}{$s$-breadth} $b_s(K,C)$ of $K$ \wrt~$C$ in the direction $s\in\R^n\setminus\set{0}$
is $b_s(K,C)=2\frac{h(K-K,s)}{h(C-C,s)}$. The $s$-breadth does not change under central
symmetrization (\cf~\cite{BrKo2}), \ie
\[ b_s(K,C)=\frac{1}{2}b_s(K-K,C)=2b_s(K,C-C)=b_s(K-K,C-C). \]
With help of the $s$-breath the diameter $D(K,C) = 2 \sup_{x,y \in K} R([x,y],C)$ of $K$ \wrt~$C$
can also be expressed as $D(K,C) = \sup_{s \in \R^n\setminus\set{0}} b_s(K,C)$ (see \cite{BrKo2}).
Moreover, denoting the norm induced by the gauge body $1/2(C-C)$ by $\norm_{1/2(C-C)}$  we also have
$D(K,C) = \max_{x,y \in K}\norm[y-x]_{1/2(C-C)}$.
However, denoting also the gauge function induced by a possibly non-symmetric $C$ by $\norm_{C}$
the diameter of $K$ may differ from $\max_{x,y \in K}\norm[y-x]_{C}$.
One may think that the length of a segment $[x,y]$ should be defined
  by the value $\norm[y-x]_C$ of the gauge function. However, we think (as it was already formulated in \cite[pg.~134]{DaGrKl})
  that this measures should
  be translation invariant and symmetric (since segments as 1-dimensional objects are always symmetric).

All three radii $r,D,R$ are non-decreasing and homogeneous of degree $1$ in the first argument, as well as
non-increasing and homogeneous of degree $-1$ in the second.
This means, \eg, that $R(K_1,C)\leq R(K_2,C)$ if $K_1\subset K_2$ and $R(\lambda K,C)=\lambda R(K,C)$,
for any $\lambda\ge 0$.
Moreover, all three radii are continuous (\wrt~the Hausdorff metric) in both arguments
and affine invariant (in the sense that, \eg~$R(A(K),A(C))=R(K,C)$, for any $n$-dimensional
regular affine transformation $A$).

The Minkowski asymmetry $s(K)$
%is the smallest dilatation factor needed to cover $-K$ by a homothetic of $K$ itself. In other words,
can be formalized as
$s(K) := \inf\set{\rho > 0 :-K \subset_t \rho K} = R(-K,K)$.
If $-(K-c) \subset s(K) (K-c)$ for some $c \in \R^n$ then $c$ is called a \cemph{dred}{Minkowski center} of $K$,
and if $c=0$ is a Minkowski center of $K$, then $K$ is called \cemph{dred}{Minkowski centered}.
Moreover, we say that  $K$ is \cemph{dred}{Minkowski concentric} (or \cemph{dred}{mirrored Minkowski concentric}) \wrt~$C$,
if there exists a Minkowski center $c$ of $C$ and a translation $t \in \R^n$ such that
$r(K,C)(C-c) \subset K-t \subset R(K,C)(C-c)$ (or $-r(K,-C)(C-c)\subset K-t \subset R(K,C)(C-c)$, respectively).
Now, consider the case that we can choose $t$ above to be a Minkowski center of $K$. Then it is easy to see that the
Minkowski concentricity of $K$ \wrt~$C$ is equivalent to the Minkowski concentricity of $C$ \wrt~$K$. In this case we simply say that
$K$ and $C$ are \cemph{dred}{Minkowski concentric}.
One should also recognize that if $K$ is mirrored Minkowski concentric \wrt~$C$ with $t$ a Minkowski center of $K$, then $-C$
is mirrored Minkowski concentric \wrt~$-K$.
 % \RSays{Could it be that $K$ is Minkowski concentric \wrt~$C$ and vice versa, but one needs different Minkowski centers?}
 % \BSays{Yes, I think that there might be such examples. In other words, I think that $K$ and $C$ Minkowski concentric is in general stronger than $K$ (and $C$) Minkowski concentric \wrt~$C$ (and $K$).}
 % \RSays{ I doubt it but this is nothing which must be in this paper.}
 %\BSays{I will try to find them :)}
The Minkowski concentricities turn out to be necessary conditions for the equality cases of many of the relevant inequalities and inequality chains in this paper.
%, \eg for
% \eqref{eq:pm-concentricity} or \eqref{eq:pp-concentricity}. Moreover, if $K$ and $C$ attain equality in \eqref{eq:sr+R<=s+12D} (or \eqref{eq:(1+s)/sR<=s+12D}),
% then $K$ and $C$ must be Minkowski concentric.}

The Minkowski asymmetry fulfills $1 \le s(K) \le n$ with equality on the
left side if and only if $K=_t-K$ and equality on the right side if and only if $K$ is an $n$-simplex
(see \cite{minkowski1897allgemeine}, or \cite{BrKo2} for a proof given in english).

In \cite{Gr} it is shown that $c$ is a Minkowski center of $K$ if and only if for all $s \in \R^n\setminus\set{0}$
it holds
\begin{equation}\label{defAsymmetry}
\frac{h(-c-K,-s)}{h(-c-K,s)}\in[s(K)^{-1},s(K)] .
\end{equation}
The geometrical meaning of \eqref{defAsymmetry} is that any hyperplane with normal $s$
containing the Minkowski center $c$ of $K$, splits $K$ in two parts with an $s$-breadth ratio bounded from above by $s(K)$.
In this sense, the Minkowski center with respect to the directional breadth
plays the role that the the centroid plays with respect to the volume (\cf~\cite{Grunb}).

A set $K$ is of \cemph{dred}{constant width} \wrt~$C$, if $D(K,C) = b_s(K,C)$
independently of the choice of $s\in\R^n\setminus\set{0}$.
It is well known that $K$ is of constant width
if and only if $K-K =D(K,C)/2(C-C)$ (see, \eg,\cite[(A)]{Egg}).

The following is a trivial but important observation: while some symmetric $C$
 do not admit non-trivial constant width sets (like parallelotopes or if they are indecomposable,
 like crosspolytopes in dimension 3 or greater),
 the \cemph{dred}{difference body} $C-C$ is
\emph{always} a non-trivial constant width body \wrt~$C$, if $C$ is non-symmetric.
%\RSays{I do not see the importance for this paper anymore}

It is well known that if $K$ is of constant width \wrt~$C$ then
it is complete \wrt~$C$. However, even though this two properties are equivalent in the Euclidean space and in any planar Minkowski
space, they are no longer equivalent for general (symmetric) $C$ if the dimension of the space is greater than 2 (see \cite{Egg}).

%In addition to the completeness definitions above we say that $K^*$ is a \cemph{dred}{Scott completion}
%of $K$ if $K^*$ is a completion of $K$ and $R(K^*,C)=R(K,C)$.
%In \cite{Sc} the existence of Scott completions is proven for the Euclidean case and
%with almost the same proof \cite{Vr} shows the same for general Minkowski spaces (\cf~\cite[Corollary 2.8]{BrGo}),
%thus for general symmetric $C$. %\RSays{-- Here it would be nice if we would know if Scott completions always exist! --} \BSays{To be checked. So far, not easy.}
%However, one should recognize that for any pair $K,C\in\CK^n$, which fulfills the right-hand inequality in \eqref{eq:sr+R<=s+12D}
%with equality, it follows from $K^*,C$ still fulfilling this inequality with equality, that
%$R(K,C)=R(K^*,C)$. 
%Hence the completions $K^*$ in this situation are always Scott completions of $K$.
%\RSays{Also Scott completions seem not to play any role here anymore}

It is shown in \cite[Lemma 3.2 (ii)]{BrGoJaMa} that we may symmetrize the gauge body $C$:

\begin{prop}\label{CcompleteC-C}
Let $K,C\in\CK^n$. Then the following are equivalent:
\begin{enumerate}[(i)]
\item $K$ is complete \wrt~$C$ and
\item $K$ is complete \wrt~$C-C$.%, and
%\item $K=\bigcap_{x\in K}(x+D(K)/2(C-C))$ (spherical intersection property).
\end{enumerate}
\end{prop}

One should recognize that $K$ is complete \wrt~$C$ does not imply that $K-K$ is complete \wrt~$C$.
Otherwise Proposition \ref{CcompleteC-C} would imply that $K-K$ is complete \wrt~$C-C$, which is only possible if
$K-K =\rho(C-C)$ for some $\rho\in\R$. However, the latter would imply that $K$ is of constant width \wrt~$C$,
but as mentioned above, completeness does not always imply constant width.

\section{New inequalities}

%Before proving Theorem \ref{th:sr+R<=s+12D} we need some technical lemmas.
The following lemma collects a series of easy to obtain inequalities,
which are used in the proofs of the main theorems below.

\begin{lem}\label{lem:Rr-sn} \label{RrC-C}
  Let $K,C\in\CK^n$. Then
  \begin{enumerate}[(a)]
  \item \label{eq:Rr-sn(i)} $\max\set{s(K),s(C)} \le \frac{R(K,C)}{r(K,-C)}$
    and $s(C)= \frac{R(K,C)}{r(K,-C)}$ implies that $K$ is mirrored
    Minkowski concentric \wrt~$C$ while
    $s(K)= \frac{R(K,C)}{r(K,-C)}$ implies that $C$ is mirrored
    Minkowski concentric \wrt~$K$.
  \item \label{(i)RrC-C} $\frac{s(C)+1}{s(C)} \le \frac{R(K,C)}{R(K,C-C)} \le s(C)+1$.
  \item \label{(ii)RrC-C} $\frac{s(C)+1}{s(C)} \le \frac{r(K,C)}{r(K,C-C)} \le s(C)+1$.
  \item \label{eq:Rr-sn(ii)}
    $\min\set{s(K),s(C)} \ge \max\set{\frac{r(K,-C)}{r(K,C)},\frac{R(K,-C)}{R(K,C)}}$.
  \item \label{RrC-C_iii} $s(C) \ge \frac{R(K,C)r(K,C-C)}{r(K,C)R(K,C-C)}$.
  \end{enumerate}
\end{lem}

\begin{proof}
  \begin{enumerate}[(a)]
  \item As a direct consequence of the definitions of the inradius and the circumradius
    we obtain
    \[\frac{r(K,-C)}{R(K,C)} (-K) \subset_t r(K,-C)(-C) \subset_t K \subset_t R(K,C) C.\]
    This implies for the Minkowski asymmetries that $s(K)=R(-K,K) \le R(K,C)/r(K,-C)$
    as well as $s(C) =R(-C,C) \le  R(K,C)/r(K,-C)$. Now, let us assume that
    $s(K)= R(K,C)/r(K,-C)$ and without loss of generality that $K$ is Minkowski centered.
    Then we immediately see
    that there exists a translation $t \in \R^n$ such that
    $-K \subset^{\opt} t+R(K,C)(-C) \subset^{\opt} s(K)K$. Hence $C$ is mirrored Minkowski concentric \wrt~$K$.
    The remaining part of the claim follows analogously.
  \item From the definitions of circumradius and Minkowski asymmetry
    it follows \[K \subset_t R(K,C-C)(C-C) \subset_t R(K,C-C)(s(C)+1)C\]
    as well as \[K \subset_t R(K,C) C \subset_t R(K,C) \frac{s(C)+1}{s(C)} (C-C).\]
    While the first chain of inclusions implies $R(K,C) \le (s(C)+1)R(K,C-C)$ the other implies
    $R(K,C-C) \le R(K,C) \frac{s(C)+1}{s(C)}$.
  \item This can be shown completely analogous to Part \eqref{(i)RrC-C} replacing the circumradii by the inradii.
  \item Since $r(K,C) = R(C,K)^{-1}$ it suffices to show $s(C) \ge \max\set{\frac{r(K,-C)}{r(K,C)},\frac{R(K,-C)}{R(K,C)}}$.
    However, from Part \eqref{(i)RrC-C} we know $\frac{R(K,-C)}{R(K,C-C)} \le s(C)+1$ and
    $\frac{R(K,C)}{R(K,C-C)} \ge \frac{s(C)+1}{s(C)}$. Dividing the second by the first inequality we obtain
    $\frac{R(K,-C)}{R(K,C)} \le s(C)$. The part of statement involving the inradius-ratio follows analogously applying both
    inequalities in Part \eqref{(ii)RrC-C}.
    % Simply applying the definitions of the inradius and the Minkowsky asymmetry one obtains
    % \[\frac{r(K,-C)}{s(C)}C \subset_t r(K,-C)(-C)\subset_t K \subset_t s(K)(-K)\]
    % and it immediately follows $r(K,-C)/s(C) \le r(K,C)$ as well as $r(K,-C) \le s(K)r(-K,-C) = s(K)r(K,C)$.
    % The statement about the outer radii follows analogously or from using the fact that $r(C,K)=R(K,C)^{-1}$.
  \item The last inequality follows directly from dividing the right inequality in Part \eqref{(i)RrC-C}
    by the left inequality in \eqref{(ii)RrC-C}.
  \end{enumerate}
\end{proof}

\begin{lem}\label{lem:breadths}
Let $C\in\CK^n$ be Minkowski centered and $r\in[0,1]$. Then
\begin{equation*}
\frac{h(C,a)+h(rC,-a)}{h(C,a)+h(C,-a)}\in\left[\frac{1+s(C)r}{1+s(C)},\frac{r+s(C)}{1+s(C)}\right], \,a\in\R^n\setminus\set{0}.
\end{equation*}
\end{lem}

\begin{proof}
Defining $x:=\frac{h(C,-a)}{h(C,a)}$ we know from \eqref{defAsymmetry} that $x \in [s(C)^{-1},s(C)]$. Using this $x$ we may
rewrite the fraction $\frac{h(C,a)+h(rC,-a)}{h(C,a)+h(C,-a)}$ dividing both enumerator and denominator by $h(C,a)$ as
$f(x):=\frac{1+rx}{1+x}$. Since $f(x)$ is a decreasing function for $r \le 1$, we conclude that
$f(x)\in[f(s(C)),f(s(C)^{-1})]$, proving the assertion.
\end{proof}

\bigskip

\begin{proof}[Proof of Theorem \ref{th:sr+R<=s+12D}]
The left inequality of both chains (since equal) as well as the middle inequality of \eqref{eq:(1+s)/sR<=s+12D}
  follow directly from Lemma \ref{lem:Rr-sn} \eqref{eq:Rr-sn(i)},
  the middle inequality of \eqref{eq:sr+R<=s+12D} directly from Lemma \ref{lem:Rr-sn} \eqref{eq:Rr-sn(ii)}.
  Moreover, the right inequality of \eqref{eq:(1+s)/sR<=s+12D} follows from \eqref{eq:ExtBohnenblust}.
Thus it only remains to show the right inequality in \eqref{eq:sr+R<=s+12D}, the generalized concentricity
inequality \eqref{eq:pp-concentricity}.
For the proof we may assume \Wlog~that $K\subset^{\opt} C \subset -s(C)C$, \ie $R(K)=1$ and $C$ is Minkowski centered.
Moreover, let $c \in C$ be such that $c+r(K,C)C \subset K$.
Since $K\subset^{\opt}C$ by Proposition \ref{prop:occ} there exist
$p^j\in \bd(K)\cap\bd(C)$, $j\in[m]$, as well as $a^j\in N(C,p^j) \setminus \set{0}$
and $\lambda_j > 0$, $j \in [m]$ such that $\sum_{j=1}^m\lambda_j=1$ and $\sum_{j=1}^m\lambda_ja^j=0$.
The latter implies the existence of some $j \in[m]$ such that $c^Ta^j \le 0$, which we may assume to be $1$.

Abbreviating $r:=r(K,C)$, we obtain from $K':=\conv(\set{p^1} \cup (c+rC)) \subset K$ and $h(K',a^1)=(p^1)^Ta^1=h(C,a^1)$ that
\begin{align*}%\label{eq:sr+Rproof}
\frac 12 D(K,C) & \ge \frac 12 b_{a^1}(K,C) = \frac{h(K,a^1)+h(K,-a^1)}{h(C,a^1)+h(C,-a^1)} \\
& \ge \frac{h(K',a^1) + h(K',-a^1)}{h(C,a^1)+h(C,-a^1)}
 = \frac{h(C,a^1)+h(rC,-a^1)-c^Ta^1}{h(C,a^1)+h(C,-a^1)}\\
\intertext{and using the assumption $c^Ta^1 \le 0$ and Lemma \ref{lem:breadths} this is}
& \ge  \frac{h(C,a^1)+h(rC,-a^1)}{h(C,a^1)+h(C,-a^1)}
 \ge \frac{1+rs(C)}{1+s(C)}.
\end{align*}
\end{proof}

Let us keep two facts, we immediately obtain from the above proof:
\begin{enumerate}[a)]
\item \label{eq:pp-c-equal} equality in the generalized concentricity inequality \eqref{eq:pp-concentricity} implies
  that $K$ is Minkowski concentric \wrt~$C$.
\item \label{eq:pm-c-equal} equality in the mirrored concentricity inequality \eqref{eq:pm-concentricity} implies
  that $s(C) \le s(K)$ and that $C$ is mirrored Minkowski concentric \wrt~$K$ (in addition to the Minkowski
    concentricity of $K$ and $C$, which one gets from the induced equality in the generalized concentricity inequality).
\end{enumerate}
Part \eqref{eq:pm-c-equal} directly follows from the equality case $s(K)= \frac{R(K,C)}{r(K,-C)}$ in
Lemma \ref{lem:Rr-sn} \eqref{eq:Rr-sn(i)}. To see that Part \eqref{eq:pp-c-equal} holds,
let us assume without loss of generality
that $C$ is Minkowski centered, that $r(K,C) C \subset K$ and $R(K,C)=1$.
Our aim is to show that $K \subset C$, thus proving the Minkowski concentricity of $K$ \wrt~$C$.
Let us observe the following: Lemma \ref{lem:breadths} shows that if $p^1 \in \bd(C)$ then
\[\frac 12 D(\conv(\set{p^1} \cup r(K,C)C),C) \ge \frac{s(C)r(K,C)+1}{s(C)+1}.\]
Hence for any $\rho >1$ and any $y \in \bd(\rho C)$ it holds
\[\frac 12 D(\conv(\set{y} \cup r(K,C) C),C) \ge \frac{s(C)r(K,C)+\rho}{s(C)+1} > \frac{s(C)r(K,C)+1}{s(C)+1}.\]
Thus equality in \eqref{eq:pp-concentricity} implies that no such $y$ exists in $K$ and therefore $K \subset C$.

% \BSays{Equality in \eqref{eq:pp-concentricity}, for some $K,C\in\CK^n$. Assuming that $r:=r(K,C)$, $R:=R(K,C)=1$, $D:=D(K,C)$, and $s:=s(C)$,
% let us suppose that $rC\subset K$, and $C$ is Minkowski centered. Our aim is to show that $K\subset C$, thus concluding the proof. Let us observe the following:
% Lemma \ref{lem:breadths} says that if $x\in\bd(C)$, then
% \[
% D(\{x\}\cup rC,C)\geq \frac{2}{s+1}(sr+1).
% \]
% In particular, for any $y\in\bd(\rho C)$, with $\rho>1$, we have that
% \[
% D(\{y\}\cup rC,C)\geq\frac{2}{s+1}(sr+\rho)>\frac{2}{s+1}(sr+1).
% \]
% Since we have equality in \eqref{eq:pp-concentricity}, we have $D(K,C)=\frac{2}{s+1}(sr+1)$, and thus we must have that $K\subset C$, thus concluding the proof.
% }
\medskip

A classical result in Euclidean geometry on simplices says that the circumradius-inradius ratio of an $n$-simplex $S$,
  is at least $n$, \ie, $R(S,\B_2^n)\geq nr(S,\B_2^n)$ (\cf~\cite[p.~28]{toth2013Lagerungen} -- with equality if and only if
  $S$ is regular). This was extended in \cite[Theorem 6.1]{BrKo2} saying the following:

  \begin{prop} \label{prop:FT} Let $K,C\in\CK^n$. Then
    \begin{equation*}
      \frac{R(K,C)}{r(K,C)} \ge \max\set{\frac{s(K)}{s(C)},{\frac{s(C)}{s(K)}}}.
    \end{equation*}
  \end{prop}

  Recognizing that Proposition \ref{prop:FT} can simply be obtained from dividing inequalities from
    Lemma \ref{lem:Rr-sn} \eqref{eq:Rr-sn(i)} and \eqref{eq:Rr-sn(ii)}, we easily see that  
  $\frac{R(K,C)}{r(K,C)} = \frac{s(K)}{s(C)}$ implies that $s(K) \ge s(C)$ as well as $C$ is
  mirrored Minkowski concentric \wrt~$K$ and vice versa that $\frac{R(K,C)}{r(K,C)} = \frac{s(C)}{s(K)}$
  implies that $s(C) \ge s(K)$ as well as $K$ is mirrored Minkowski concentric \wrt~$C$.

  Surely, in general no constant exists bounding the circumradius-inradius ratio from above.
  However, the family of complete bodies allows such an upper bound:

  \begin{lem}\label{thm:rFT}
    Let $K,C\in\CK^n$. If $K$ is complete \wrt~$C$ then
    \begin{equation*} %\label{eq:rFT}
      \frac{R(K,C)}{r(K,C)} \le s(K)s(C)
    \end{equation*}
    and equality implies that $K$ and $C$ are Minkowski concentric.
\end{lem}

\begin{proof}
Let $K$ be complete \wrt~$C$. Then $K$ is complete \wrt~$C-C$ by Proposition \ref{CcompleteC-C} and therefore
fulfills $R(K,C-C)/r(K,C-C) = s(K)$ by \cite[Theorem 1.2]{BrGo}.
Now the claimed inequality follows directly using Lemma \ref{RrC-C} \eqref{RrC-C_iii}.
For the concentricity statement let us assume without loss of generality that $C$ is Minkowski centered.
  It is shown in \cite{BrGo} that $K$ is complete \wrt~$C-C$ implies that
  $r(K,C-C)(C-C) \subset K-c \subset R(K,C-C)(C-C)$ for some Minkowski center $c$ of $K$, thus
  $K$ and $C-C$ are Minkowski concentric.
  However, $\frac{R(K,C)}{r(K,C)} \le s(K)s(C)$ now implies by using Lemma \ref{RrC-C} \eqref{RrC-C_iii} again that
  \begin{equation*}
    \begin{split}
      & r(K,C)C=\left(1+\frac{1}{s(C)}\right)r(K,C-C)C\subset r(K,C-C)(C-C)\subset K-c \subset \\
      & R(K,C-C)(C-C) \subset (1+s(C))R(K,C-C)C=R(K,C)C,
    \end{split}
  \end{equation*}
  which shows $K$ and $C$ are Minkowski concentric.

\end{proof}

  % \BSays{Let $K$ be complete \wrt~$C$ with $R(K,C)/r(K,C)=s(K)s(C)$, and let us suppose that $C$ is Minkowski centered.
  % Since $K$ is also complete
  % \wrt~$C-C$, we have that $r(K,C-C)(C-C)\subset K-t\subset R(K,C-C)(C-C)$ where $t$ is a Minkowski center of $K$.
  % Moreover,
  % \begin{equation*}
  % \begin{split}
  % & r(K,C)C=\left(1+\frac{1}{s(C)}\right)r(K,C-C)C\subset r(K,C-C)(C-C)\subset K-t\subset \\
  % & R(K,C-C)(C-C) \subset (1+s(C))R(K,C-C)C=R(K,C)C,
  % \end{split}
  % \end{equation*}
  % hence $K$ is Minkowski concentric \wrt~$C$, and thus $K$ and $C$ are Minkowski concentric.
  % }

See Example \ref{exa:niceexample} below for sets $K,C$ such that $K$ is complete \wrt~$C$ and
  $\frac{R(K,C)}{r(K,C)} =s(K)s(C)$ for any presribed values for $s(K)$ and $s(C)$.

\begin{proof}[Proof of Theorem \ref{thm:complete-chain}]
This theorem now follows directly from combining Theorem \ref{th:sr+R<=s+12D} and Lemma \ref{thm:rFT}.
\end{proof}

Putting together the concentricity statements after the proof of Theorem \ref{th:sr+R<=s+12D} and from Lemma \ref{thm:rFT}, we see that
equality in the complete chain of inequalities \eqref{eq:complete chain} implies that $K$ and $C$ are Minkowski concentric as well as that
$C$ is mirrored Minkowski concentric \wrt~$K$.

%\BSays{Equality above implies (A) $K$ and $C$ are Minkowiski concentric, and (B) $C$ is Mirrored Minkowski concentric \wrt~$K$.
%A long algebraic computation shows that if $K$ would be M.M.C. \wrt~$C$, then we might have that $\frac{R(K,C)}{R(-K,C)}=s(K)$. Whether this is true or not, is not clear to me.}

  \begin{rem} \label{rem:pseudo}
    For a general pair $K,C \in \CK^n$ (not necessarily such that $K$ is complete \wrt~$C$) it always holds
    \begin{equation}\label{eq:ExtendedJung}
      \frac{s(K)+1}{s(K)} K \subset K-K  \subset \frac{D(K,C)}{2}(C-C) \subset_t \frac{D(K,C)}{2}(s(C)+1) C
    \end{equation}
    %(see \cf~\cite[Theorem 4.1]{BrKo2})
    and the following are equivalent:
    \begin{enumerate}[(i)]
    \item all inequalities in the inequality chain \eqref{eq:complete chain} are fulfilled with equality and
    \item $\frac{D(K,C)}{2}(s(C)+1) C \subset_t (s(K)+1) (-K)$.
    \end{enumerate}
  \end{rem}

\section{Complete simplices}

An $n$-simplex $S\in\CK^n$, $S=\conv(\set{p^1,\dots,p^{n+1}})$, is \cemph{dred}{equilateral} \wrt~$C\in\CK^n$
if $D([p^i,p^j],C)=D(S,C)$ for any $1\leq i<j\leq n+1$. The following remark shows
that a complete simplex is equilateral.

\begin{rem}
Let $S,C\in\CK^n$ be such that $S=\conv(\set{p^1,\dots,p^{n+1}})$ is an $n$-simplex being complete \wrt~$C$.
Then for every boundary point $p$ of $S$ and the vertex $p^j$ not belonging to the facet of $S$ to which $p$ belongs, it holds
$D([p,p^j],C)=D(S,C)$ and it follows that $S$ is equilateral.
\end{rem}

\begin{proof}
  For better readability we write $d(x,y):=D([x,y],C)=||x-y||_{1/2(C-C)}$.
  Since $S$ is complete, every boundary point of $S$ must be an endpoint of a diametrical segment of $S$ \cite{Egg}
  and surely
  the vertex $p^j$ not belonging to the facet of $S$ to which $p$ belongs may be chosen as the other endpoint.
  Now, consider the point $p:=\sum_{k \in [n+1] \setminus \set{j}} \frac{1}{n} p^k \in \relint(\conv(\set{p^k: k \neq j}))$,
  which is an endpoint of a diametrical segment such that the other endpoint must be $p^j$ implying $d(p,p^j)=D(S,C)$.
  Hence
  \begin{equation*}
    % \begin{split}
    D(S,C) =d(p,p^j)=d(\sum_{k \in [n+1] \setminus \set{j}} \frac{1}{n} p^k, p^j ) \le \frac{1}{n}
    \sum_{k \in [n+1] \setminus \set{j}} d(p^k,p^j) \le D(S,C),
    % \end{split}
  \end{equation*}
  which in particular implies that $d(p^k,p^j)=D(S,C)$ for every $j \in [n+1]$ and every $k \in [n+1] \setminus \set{j}$.
\end{proof}

We now prove Theorem \ref{th:Schain}, characterizing the equality cases of \eqref{eq:BrKo1}
by connecting it to the equality cases of \eqref{eq:sr+R<=s+12D} and to the notion of completeness.

\begin{proof}[Proof of Theorem \ref{th:Schain}]
Since all the statements, which we want to show to be equivalent are invariant \wrt~a simultaneous
  affine transformation of $S$ and $C$ we may suppose \Wlog~that $S=\conv(\set{p^1,\dots,p^{n+1}})$ with $\norm[p^j]_2=1$, $j\in[n+1]$, is a regular simplex (with respect to the Euclidean norm) centered in $0$.
%\begin{equation}\label{eq:ExtendedJung}
%  \frac{n+1}{n} S \subset S-S  \subset \frac{D(K,C)}{2}(C-C) \subset_t \frac{D(K,C)}{2}(s(C)+1)C
%\end{equation}
%(see \cf~\cite[Theorem 4.1]{BrKo2}).

Since \eqref{(iii)schain} and \eqref{(iv)schain} are due to Remark \ref{rem:pseudo} anyway equivalent it suffices to show
  the following line of indications: \eqref{(i)schain} $\Rightarrow$ \eqref{(ii)schain} $\Rightarrow$ \eqref{(iii)schain},
  \eqref{(iv)schain} $\Rightarrow$ \eqref{(i)schain}, and \eqref{(iii)schain},\eqref{(iv)schain} $\Rightarrow$ \eqref{(v)schain}
  $\Rightarrow$ \eqref{(ii)schain}.

%We prove the claimed equivalences in showing the following line of indications:
%\eqref{(i)schain} $\Rightarrow$ \eqref{(ii)schain} $\Rightarrow$ \eqref{(iii)schain} $\Rightarrow$ \eqref{(iv)schain} $\Rightarrow$ \eqref{(i)schain}
%and \eqref{(iv)schain}-\eqref{(iii)schain} $\Leftrightarrow$ \eqref{(v)schain}.

\begin{enumerate}[\enquote{(i) $\Rightarrow$ (ii)}]

\item[\enquote{\eqref{(i)schain} $\Rightarrow$ \eqref{(ii)schain}}]
  Plugging \eqref{eq:BrKo1} in \eqref{(i)schain} we obtain that $R(S,C)\leq nr(-S,C)$,
  which together with Lemma \ref{lem:Rr-sn} \eqref{eq:Rr-sn(i)} implies
  $R(S,C)=nr(-S,C)$ and then also equality in \eqref{eq:BrKo1}
   as well, which proves \eqref{(ii)schain}.

\item[\enquote{\eqref{(ii)schain} $\Rightarrow$ \eqref{(iii)schain}}]

 \eqref{(ii)schain} implies optimality in all inclusions in \eqref{eq:ExtendedJung}. In particular we have
 \[ \frac {n+1}{n} S\subset S-S\subset^{\opt}_t\frac{D(S,C)}{2}(s(C)+1)C.\] %\RSays{the  second should be $\subset^{\opt}_t$, or?}
 Moreover, since the only outer normal of $S-S$ at the vertex $(n+1)/n p^i$ of $(n+1)/n S$ is $p^i$, $i\in[n+1]$,
 this $p^i$ is also the (only) outer normal of $\frac{D(S,C)}{2}(s(C)+1)C$ at $(n+1)/n p^i$, $i\in[n+1]$.
 Hence
 \[ \frac{D(S,C)}{2}(s(C)+1)C \subset \bigcap_{j=1}^{n+1} \set{x\in\R^n:x^Tp^j \le (\frac{n+1}n p^j)^Tp^j} = (n+1)(-S). \]

% \eqref{(ii)schain} implies that for each two terms in the inclusion chain in \eqref{eq:ExtendedJung} it holds that the more left
% is optimally contained in the other.
% Hence $(n+1)/n S\subset^{\opt}\frac{D(S,C)}{2}(s(C)+1)C$, which means there must exist
% \[a^j\in N((n+1)/n S,(n+1)/n p^j) \cap N(\frac{D(S,C)}{2}(s(C)+1)C,(n+1)/n p^j),\]
% with $\norm[a^j]_2=1$, $j\in[n+1]$. \RSays{ -- that we need all $n+1$ follows from sum $a^j=0$, this should
%   be mentioned!}
% Since $(n+1)/n S\subset S-S$, the outer normals are uniquely determined by the outer normal of the facets of $S$,
% which means we have $a^j=p^j$, $j\in[n+1]$.
% Hence
% \[ \frac{D(S,C)}{2}(s(C)+1)C \subset \bigcap_{j=1}^{n+1} \set{x\in\R^n:x^Tp^j \le ((n+1)/n p^j)^Tp^j} = (n+1)(-S). \]

%\item[\enquote{\eqref{(iii)schain} $\Rightarrow$ \eqref{(iv)schain}}]

%  Assuming \eqref{(iii)schain} we immediately obtain from the rightmost (optimal) containment that
%  $r(-S,C) = \frac{s(C)+1}{n+1}\frac{D(S,C)}2$ and thus equality from left to right in both chains \eqref{eq:sr+R<=s+12D} and
%    \eqref{eq:(1+s)/sR<=s+12D}.
\item[\enquote{\eqref{(iv)schain} $\Rightarrow$ \eqref{(i)schain}}]
  %In general we have $1/R(S,C) S \subset C \subset 1/r(-S,C)(-S)$ and thus $n=s(S) \le R(S,C)/r(-S,C)$.
 % \RSays{\sout{Assuming \eqref{(iv)schain} to be true and taking \eqref{eq:sr+R<=s+12D} and
     % Lemma \ref{lem:Rr-sn} \eqref{eq:Rr-sn(i)} into account, we obtain
     % \begin{equation*}
     %   \begin{split}
     %     r(S,-C)+R(S,C) & \le s(C)r(S,C)+R(S,C) \le \frac{s(C)+1}{2}D(S,C)\\
     %     & = (n+1)r(S,-C) \le r(S,-C)+R(S,C)
     %   \end{split}
     % \end{equation*}
     % and therefore equality in all inequalities above.}
    Equality in the full chain \eqref{eq:sr+R<=s+12D} implies equality in the generalized concentricity inequality
    \eqref{eq:pp-concentricity}, which is part of the chain.
\item[\enquote{\eqref{(iii)schain}-\eqref{(iv)schain} $\Rightarrow$ \eqref{(v)schain}}]

  On the one hand \eqref{(iii)schain} implies by Proposition \ref{prop:ScompleteC-C} that $S$ is complete \wrt~$C-C$,
  and thus by Proposition \ref{CcompleteC-C} also completeness of $S$ \wrt~$C$.
  On the other hand \eqref{(iv)schain} implies
  $s(C)r(S,C)+R(S,C)=\frac{(n+1)R(S,C)}{n}$ and therefore $R(S,C)/r(S,C)=ns(C)$.

\item[\enquote{\eqref{(v)schain} $\Rightarrow$ \eqref{(ii)schain}}]%-\eqref{(iv)schain}}]

  %By virtue of \eqref{(i)schain} $\Rightarrow$ $\dots$ $\Rightarrow$ \eqref{(iv)schain} $\Rightarrow$ \eqref{(i)schain},
  % Surely, it is now enough to show \eqref{(v)schain} $\Rightarrow$ \eqref{(ii)schain}.
  Assuming that $S$ is complete \wrt~$C$ it follows from Proposition \ref{CcompleteC-C} that
  $S$ is also complete \wrt~$C-C$.
  Now, using %\sout{\cite[Theorem 1.2]{BrGo}}
  the equivalences in Proposition \ref{prop:ScompleteC-C} % (iv) and \eqref{eq:chainMinkowski}}
  %\RSays{-- shouldn't we give this as a Proposition before?--}
  we obtain \[(n+1)r(S,C-C)=\frac{n+1}{n}R(S,C-C)=D(S,C-C).\]
  Moreover, $R(S,C)=ns(C)r(S,C)$ means equality in Lemma \ref{RrC-C} \eqref{RrC-C_iii}, which
  implies equality in Lemma \ref{RrC-C} \eqref{(i)RrC-C}, thus $R(S,C)=(s(C)+1)R(S,C-C)$.
  Finally, since $D(S,C)=2D(S,C-C)=2(n+1)r(S,C-C)$, we conclude
  \[\frac{R(S,C)}{D(S,C)}=\frac{(s(C)+1)R(S,C-C)}{2(n+1)r(S,C-C)}=\frac{(s(C)+1)n}{2(n+1)}.\]

\end{enumerate}
\end{proof}

\begin{cor}
Let $S,C\in\CK^n$ be such that $S$ is an $n$-simplex. Then $S$ is complete \wrt~$C$ if and only if $-S$ is complete \wrt~$C$,
\[\frac{n}{s(C)} \leq \frac{R(\pm S,C)}{r(\pm S,C)} \leq n s(C),\]
and $S$ (resp.~$-S$) attains equality in the right-hand side if and only if $-S$ (resp.~$S$) attains equality in the
left-hand side if and only if $S$ (resp.~$-S$) fulfills any/all conditions of Theorem \ref{th:Schain}.
\end{cor}

\begin{proof}
The equivalence of the completeness of $S$ and $-S$ follows directly from Proposition \ref{CcompleteC-C}, while the
  inequality chain is just a combination of Proposition \ref{prop:FT} and Lemma \ref{thm:rFT}.
  Moreover, equality on the right means that Condition \eqref{(v)schain} of Theorem \ref{th:Schain} is fulfilled and thus
  $S$ fulfills all the conditions in Theorem \ref{th:Schain}.

Now, let us suppose that $S$ attains equality on the left-hand side.
It implies that the chain of inclusions
\begin{equation*}
  \begin{split}
    S\subset_tR(S,C)C\subset_tR(S,C)s(C)(-C)\subset_t\frac{R(S,C)s(C)}{r(S,C)}(-S)
  \end{split}
\end{equation*}
possesses optimal containment of the first in the last set, and thus of any set in any other including set in the chain.
Now, we obtain from the optimal containment of the first in the third set that
$R(-S,C)=s(C)R(S,C)$, whereas from the second in the fourth that $r(-S,C)=r(S,C)/s(C)$. Hence
$R(-S,C)/r(-S,C)=s(C)^2R(S,C)/r(S,C)=s(C)n$.
Finally, because $-S$ is complete \wrt~$C$ if and only if  $S$ is complete \wrt~$C$ we obtain that $S$ attains equality on the left-hand side means that
$-S$ fulfills Condition \eqref{(v)schain} of Theorem \ref{th:Schain}.
\end{proof}

%\RSays{\sout{
%\begin{proof}[Proof of Theorem \ref{eq:complete chain}]
%In view of Theorem \ref{th:sr+R<=s+12D}, it is enough to show that
%\[
%\frac{1+s(K)}{s(K)}R(K,C)\leq s(C)r(K,C)+R(K,C).
%\]
%The latter is equivalent to $R(K,C)\leq s(K)s(C)r(K,C)$, which follows from Lemma \ref{thm:rFT} because
%$K$ is complete \wrt~$C$.
%\end{proof}}}

The following example presents a particular family of pairs $S,C\in\CK^n$, $S$ being an $n$-simplex, which fulfill
  the conditions of Theorem \ref{th:Schain}.

\begin{example}\label{exa:niceexample}
Let $\lambda \ge \mu\ge 0$ and $S,C\in\CK^n$ such that $S$ is a Minkowski-centered equilateral (\wrt $\B_2^n$)
$n$-simplex and $\lambda S+\mu(-S)\subset C \subset (\lambda+2\mu)S\cap(2\lambda+\mu)(-S)$.

Then $S$ and $-S$ are complete \wrt~$C$, while $-S$ is fulfills the conditions of Theorem \ref{th:Schain},
whereas $S$ does not fulfill these conditions, unless $\lambda \in \set{0,\mu}$.
\end{example}

%\BSays{One should recognize that suitable rescalings of $\pm S$ and $\pm C_i$, $i=1,2$, are optimally contained on each other
%without translations. Hence, $S$ and $C$ are Minkowski concentric and mirrored Minkowski concentric \wrt~each other.}
One should recognize that by construction of $C$ it is quite obvious that
  $S$ and $C$ are Minkowski concentric and (mutually) mirrored Minkowski concentric.

\begin{proof}
Let $p^i\in\R^n$, $i\in[n+1]$ be such that $S:=\conv(\{p^i:i\in[n+1]\})$. Then
$(\lambda/n+\mu)(-p^i),(\lambda+\mu/n)p^i\in\bd(C_1)\cap\bd(C_2)$, and it follows
\begin{equation*}
\begin{split}
&R(S,C_i)=R(S,C)=\left(\lambda+\frac{\mu}{n}\right)^{-1},\quad R(-S,C_i)=R(-S,C_i)=\left(\frac{\lambda}{n}+\mu\right)^{-1},\\
&r(S,C_i)=r(S,C)=(\lambda+n\mu)^{-1},\quad r(-S,C_i)=r(-S,C)=(n\lambda+\mu)^{-1},\\
&D(S,C_i)=D(S,C)=D(-S,C)=\frac{2}{\lambda+\mu},\text{ and }s(C_i)=s(C)=R(C,-C)=\frac{n\lambda+\mu}{\lambda+n\mu}.
\end{split}
\end{equation*}
Therefore $\frac{R(-S,C)}{D(-S,C)}=\frac{n(s(C)+1)}{2(n+1)}$ (thus $-S$ fulfills the conditions of Theorem \ref{th:Schain}) but
$\frac{R(S,C)}{D(S,C)}=\frac{n(\lambda+\mu)}{2(n\lambda+\mu)}<\frac{n(s(C)+1)}{2(n+1)}$ whenever $\lambda>\mu>0$.
\end{proof}

Let us recall that in the case of Minkowski spaces, \ie~when $C=-C$, an $n$-simplex $S$ is complete
if and only if $S$ attains equality in \eqref{eq:chainMinkowski} and/or \eqref{eq:Bohnenblust}
(see Proposition \ref{prop:ScompleteC-C}).

The following example presents two sets $S,C\in\CK^n$, $n\geq 3$ such that $S$
is an $n$-simplex which is complete \wrt~$C$,
but such that neither, $S$ nor $-S$, fulfills the conditions of Theorem \ref{th:Schain}.

\begin{example}\label{example:ScompleteBUTnotpseudocomplete}
  Let $S$ be a Minkowski centered
    $n$-simplex, $p\in(n+1)(S\cap(-S))\setminus(S-S)$ and $C = \conv(\set{p} \cup (S-S))$.
  Then
  \begin{enumerate}[(i)]
  \item both, $S$ and $-S$, are complete \wrt~$C$ but
  \item $C$ is not Minkowski concentric \wrt~$S$ nor $-S$.
  \end{enumerate}
\end{example}

\begin{proof}
  \begin{enumerate}[a)]
  \item Since $C \subset (n+1) (S\cap(-S))$, we have that
    \[(S-S) \subset 1/2(C-C) \subset (n+1)(S\cap(-S)) + (n+1)(-(S\cap(-S)))=(n+1)(S\cap(-S)),\]
    and thus using Proposition \ref{prop:ScompleteC-C} together with Proposition \ref{CcompleteC-C} we see that
    $S$ and $-S$ are complete \wrt~$C$ (and also that $D(S,C)=1$).
  \item Since
    \begin{align*}
      \frac{n+1}{n}S \subset C & = \conv(\{p\}\cup(S-S))\subset (n+1)(S\cap(-S))\subset (n+1)(-S) \\
      \intertext{as well as}
      \frac{n+1}{n}(-S)\subset C & =\conv(\{p\}\cup(S-S))\subset(n+1)(S\cap(-S))\subset (n+1) S
    \end{align*}
    and because the extreme sides of each chain show optimal containment, all inclusions are optimal, \ie
    $C\subset^{\opt}(n+1)(\pm S)\subset^{\opt}nC$.
    Now, since $S-S$ is 0-symmetric, $C = \conv(\set{p}\cup(S-S))$
      cannot have 0 as a Minkowski center because of Proposition \ref{prop:occ}.
    % Now, let us assume without loss of generality that $S=\conv(\set{p^1,\dots,p^{n+1}})$ is regular
    %   (in the euclidean sense) with $p^1,\dots,p^{n+1} \in \R^n$ such that  $\norm[p^i]_2=1$, $i \in [n+1]$ and
    %   that $p^Tp^1>\frac{n+1}{n}$ (the reason that $p \in (S\cap(-S))\setminus(S-S)$).
    %   Then defining $\rho:=\frac{n}{n+1}p^Tp^1$ and using that $S$ is regular it holds
    %   $C\subset\rho(-C)$ as well as $C\cap\bd(\rho(-C)) = \set{p}$.
    % \BSays{((In details:
    % if $p\in H^>_{p^1,\frac{n+1}{n}}$, where $H^=_{p^1,\frac{n+1}{n}}$ supports $S-S$ in a facet, then
    % we also have that $H^{=}_{-p^1,\frac{n+1}{n}}$ supports $C$ (and $S-S$) at another facet.
    % Then we have a point $q\in H^{=}_{p^1,\frac{n+1}{n}}$ such that $\rho q=p$, since $(\rho q)^Tp^1
    % =\rho \frac{n+1}{n}=p^Tp^1$.))}
    % Using Proposition \ref{prop:occ}, we see that $C$ is not optimally contained in $\rho(-C)$ and
    % therefore $0$ is not a Minkowski center of $C$.
    However, as $0$ is the only Minkowski center of $\pm S$, we conclude that $C$ is not Minkowski concentric \wrt~$S$ nor $-S$.
  \end{enumerate}

    %We will show that $R(\pm S,C)/r(\pm S,C)=n$ and thus that \eqref{(v)schain} of Theorem \ref{th:Schain}
    %    is violated by both, $S$ and $-S$, since $C$ is obviously not symmetric.
    %  Since
    %  \begin{align*}
    %    \frac{n+1}{n}S \subset C & = \conv(\{p\}\cup(S-S))\subset (n+1)(S\cap(-S))\subset (n+1)(-S) \\
    %    \intertext{as well as}
    %    \frac{n+1}{n}(-S)\subset C & =\conv(\{p\}\cup(S-S))\subset(n+1)(S\cap(-S))\subset (n+1) S
    %  \end{align*}
    %  and since the extreme sides of each chain show optimal containment, all containments are optimal. Thus
    %  we directly obtain
    %  \[\frac{R(S,C)}{r(S,C)}=n\quad\text{and}\quad\frac{R(-S,C)}{r(-S,C)}=n.\]

\end{proof}

An example of a simplex $S$ and a body $C$ such that $S$ is complete \wrt $C$ and
  all the necessary Minkowski concentricity conditions between $S$ and $C$ are fufilled,
  but still $n/s(C) < R(S,C)/r(S,C) < ns(C)$ is not known to us. Thus it is possible that completeness
  together with the necessary Minkowski concentricity conditions implies all the properties in
  Theorem  \ref{th:Schain}.

\medskip

For the remainder of this section we consider the situation in the planar case.
However, doing so we make use of the following classical result on polytopes in arbitrary dimensions(\cf~\cite{schneider2013convex}),
where $\mathrm{vol}_{n-1}(\cdot)$ denotes the $(n-1)$-dimensional volume.

\begin{prop}\label{prop:Minkowski}
  Let $P\in\CK^n$ be a polytope with facets $F_i=P\cap H_{a^i,c_i}$, for some $a^i\in\R^n$, $\norm[a^i]_2=1$, $c^i\in\R$, $i\in[m]$, and $m\in\N$. Then
  \[\sum_{i=1}^m \mathrm{vol}_{n-1}(F_i)a^i=0.\]
\end{prop}

\begin{lem}\label{lem:C-C=S-S}
Let $S,C\in\CK^2$ be such that $S$ is a Minkowski centered triangle. Then the  following are equivalent:
\begin{enumerate}[(i)]
\item $S-S=C-C$ (and thus $S$ is of constant width \wrt~$C$) %=3(S \cap (-S))$.
\item $C =_t \lambda S+ (1-\lambda)(-S)$, for some $\lambda \in [0,1]$.
\end{enumerate}
\end{lem}

\begin{proof}
Surely, (ii) implies $C-C = (\lambda S+(1-\lambda)(-S))-(\lambda S+(1-\lambda)(-S))=S-S$ and therefore (i).

To prove that (i) implies (ii) we  may assume \Wlog~that $S$ is an equilateral triangle \wrt~$\B_2^2$.
Since $C-C=S-S$, $C$ must be a polygon with edges parallel to edges of $S-S$. Let $l^j$ be the length of the edge
of $C$ with outer normal $(\cos(-\pi/2+j\pi/3),\sin(-\pi/2+j\pi/3))$, $j\in\set{0,\dots,5}$.
The length of the edge of $C-C$ with outer normal $(0,-1)$ equals $l_1+l_4$,
with $(\sqrt{3}/2,-1/2)$ equals $l_2+l_5$, and with $(\sqrt{3}/2,1/2)$ equals $l_3+l_6$, and all
of them equal the length of any edge of $S-S$, hence
\begin{equation}\label{system1}
l_1+l_4=l_2+l_5=l_3+l_6.
\end{equation}
By Proposition \ref{prop:Minkowski} we have that
\[\sum_{j=1}^6l_j(\cos(-\pi/2+j\pi/3),\sin(-\pi/2+j\pi/3))=0,\]
which implies that $(l_1-l_4)(0,-1)+(l_2-l_5)(\sqrt{3}/2,-1/2)+(l_3-l_6)(\sqrt{3}/2,1/2)=0$,
and therefore
\begin{equation}\label{system2}
l_2-l_5+l_3-l_6=0 \quad \text{and} \quad l_1-l_4+l_2-l_5=0.
\end{equation}
The solution to the linear system \eqref{system1} and \eqref{system2} is $l_1=l_3=l_5$ and $l_2=l_4=l_6$, and thus
$C=l_1 S+ l_2(-S)$, with $l_1+l_2=1$.
\end{proof}

Particularizing Theorem \ref{th:Schain} to the planar case (where completeness and constant width coincide \cite{Egg}), we
see that the situation is much simpler:
%if $S,C\in\CK^2$ under the conditions of Theorem \ref{th:Schain}, then $C=_t\lambda S+\mu(-S)$ for some $\mu\geq\lambda\geq 0$. In particular, whenever $S$ is complete, then $S$ or $-S$ is under the conditions of Theorem \ref{th:Schain}.}

\begin{cor}\label{lem:Scompletepseudocomplete}
  Let $S,C\in\CK^2$ be such that $S$ is a Minkowski centered triangle. Then the following are equivalent:
  \begin{enumerate}[(i)]
  \item $3/2 S \subset S-S = \frac{D(S,C)}{2}(C-C)\subset (s(C)+1)\frac{D(S,C)}{2}C \subset 3(-S)$,
  \item $3r(-S,C) = r(-S,C)+R(S,C)= 3/2 R(S,C) = s(C)r(S,C)+R(S,C)=(s(C)+1)\frac{D(S,C)}{2}$,
  \item $s(C)r(S,C)+R(S,C)=(s(C)+1)\frac{D(S,C)}{2}$,
  \item $j(S,C)=(s(C)+1)/3$,
  \item $S$ is of constant width \wrt~$C$ and $R(S,C)=2s(C)r(S,C)$,
  \item $S$ is of constant width \wrt~$C$ and $j(S,C) \ge j(-S,C)$, and
  \item $C =_t \lambda S + (1-\lambda)(-S)$ for some $\lambda \in [0,1/2]$.
  \end{enumerate}
\end{cor}

\begin{proof}
The equivalences of (i) to (v) follow directly from particularizing Theorem \ref{th:Schain} to the planar case.
Hence it suffices to show (iv) $\Rightarrow$ (vi) $\Rightarrow$ (vii) $\Rightarrow$ (i)-(v).

If (iv) holds, $S$ attains equality in \eqref{eq:BrKo1} and therefore $j(-S,C) \le j_C = j(S,C)$.

Assuming (vi) to be true, $S$ is of constant width \wrt~$C$ and therefore
$S-S = D(S,C-C)(C-C)$. Now, Lemma \ref{lem:C-C=S-S} implies
that $D(S,C-C)C =_t \lambda S+(1-\lambda)(-S)$, for some $\lambda \in [0,1]$ and using the computations
of the radii of $\pm S$ \wrt~$C$ in Example \ref{exa:niceexample}, we see that
\[\frac{1}{2-\lambda}=\frac{R(-S,C)}{D(-S,C)}=j(-S,C)\leq j(S,C)=\frac{R(S,C)}{D(S,C)}=\frac{1}{1+\lambda}\]
implies $\lambda \in [0,1/2]$, completing (vii).

Finally, if (vii) holds true, (i)-(v) follow directly from the computations of the radii of $\pm S$ \wrt~$C$ in Example \ref{exa:niceexample}.
\end{proof}

\bibliographystyle{plain}
\bibliography{mybib}

\end{document}